\documentclass[reqno]{amsart}
\usepackage[latin1]{inputenc}
\usepackage[pdftex]{graphicx}
\usepackage{graphics}
\usepackage{graphicx}
\usepackage{epstopdf}
\usepackage[pdftex,hyperindex]{hyperref}
\usepackage{floatflt}
\usepackage{multicol}
\usepackage[T1]{fontenc}
\usepackage{textcomp}
\usepackage{latexsym}
\usepackage{expdlist}
\usepackage{vmargin}
\usepackage{booktabs}
\usepackage{placeins}
\usepackage{array}
\usepackage{amsmath}
\usepackage{amssymb}
\usepackage{amsfonts}
\usepackage{verbatim}
\usepackage{array}
\usepackage{framed}
\usepackage{amsthm}
\usepackage{enumerate}
\usepackage{cite}
\input amssym.def
\input amssym
\theoremstyle{plain}
\newtheorem{theorem}{Theorem}[section]
\newtheorem{lemma}{Lemma}[section]

\theoremstyle{remark}
\newtheorem{remark}{Remark}[section]

\theoremstyle{definition}

\newcommand\id{\mathrm{id}}

\allowdisplaybreaks[1]

\begin{document}

\title{A new class of orthogonal polynomials}

\author{Stefan Kahler}
\address{Stefan Kahler, Fachgruppe Mathematik, RWTH Aachen University, Pontdriesch 14-16, 52062 Aachen, Germany}
\email{kahler@mathematik.rwth-aachen.de}

\author{Josef Obermaier}
\address{Josef Obermaier, Department of Mathematics, Technical University of Munich, Boltzmannstra{\ss}e 3, 85748 Garching, Germany}
\email{josef.obermaier@tum.de}

\date{\today}

\begin{abstract}
We consider random walk polynomial sequences $(P_n(x))_{n\in\mathbb{N}_0}\subseteq\mathbb{R}[x]$ given by recurrence relations of the form $P_0(x)=1$, $P_1(x)=x$ and $x P_n(x)=a_n P_{n+1}(x)+c_n P_{n-1}(x)\;(n\in\mathbb{N})$, where $a_n$ and $c_n$ are positive and sum up to $1$. $(P_n(x))_{n\in\mathbb{N}_0}$ is said to satisfy nonnegative linearization of products if the product of any two polynomials $P_m(x)$, $P_n(x)$ is a convex combination of $P_{|m-n|}(x),\ldots,P_{m+n}(x)$. This property gives rise to a hypergroup structure and a sophisticated harmonic analysis. We are interested in examples such that both the original sequence $(P_n(x))_{n\in\mathbb{N}_0}$ and the sequence $(\widetilde{P_n}(x))_{n\in\mathbb{N}_0}$ which corresponds to switched roles of $(a_n)_{n\in\mathbb{N}}$ and $(c_n)_{n\in\mathbb{N}}$ satisfy nonnegative linearization of products. Such considerations were recently started by Lasser and Obermaier and can be motivated from a harmonic analytic, combinatorial or probabilistic point of view. However, Lasser and Obermaier left open the question whether examples besides the trivial example of the Chebyshev polynomials of the first kind $(T_n(x))_{n\in\mathbb{N}_0}$ (with $a_n\equiv c_n\equiv1/2$) actually exist. We provide a sufficient criterion and explicitly construct such nontrivial examples. Moreover, we provide characterizations of $(T_n(x))_{n\in\mathbb{N}_0}$ by additionally involving properties of the duals and Haar measures. Our criterion also enables us to solve open problems concerning the Haar measure of polynomial hypergroups stated by Kahler and Szwarc.
\end{abstract}



\maketitle

\numberwithin{equation}{section}

\section{Introduction}\label{sec:intro}

\subsection{Basic setting and problem under consideration}

We consider sequences $(P_n(x))_{n\in\mathbb{N}_0}\subseteq\mathbb{R}[x]$ of orthogonal polynomials which satisfy recurrence relations of the form $P_0(x)=1$, $P_1(x)=x$,
\begin{equation}\label{eq:threetermrec}
x P_n(x)=a_n P_{n+1}(x)+c_n P_{n-1}(x)\;(n\in\mathbb{N})
\end{equation}
with $(c_n)_{n\in\mathbb{N}}\subseteq(0,1)$ and $a_n\equiv1-c_n$. In the literature, such polynomials are frequently called (symmetric) `random walk polynomial sequences' (in the following `RWPS'). The theory of orthogonal polynomials, particularly Favard's theorem, yields that $(P_n(x))_{n\in\mathbb{N}_0}$ is orthogonal w.r.t. a unique probability (Borel) measure $\mu$ on $\mathbb{R}$ with $|\mathrm{supp}\;\mu|=\infty$ and $\mathrm{supp}\;\mu\subseteq[-1,1]$. We refer to \cite{Ch78} concerning standard results on orthogonal polynomials and to \cite{CSvD98,KM59,vDS93} concerning the relation to random walks (with regard to this relation, a few more comments can also be found in the following subsection). It is clear that an RWPS is normalized by $P_n(1)\equiv1$. We shall also make use of the orthonormal polynomials $(p_n(x))_{n\in\mathbb{N}_0}$ (with positive leading coefficients) which correspond to $(P_n(x))_{n\in\mathbb{N}_0}$. They are (well-known to be) given by $p_0(x)=1$, $p_1(x)=x/\alpha_1$ and $x p_n(x)=\alpha_{n+1}p_{n+1}(x)+\alpha_n p_{n-1}(x)\;(n\in\mathbb{N})$, where
\begin{equation*}
\alpha_n=\begin{cases} \sqrt{c_1}, & n=1, \\ \sqrt{c_n a_{n-1}}, & n\geq2. \end{cases}
\end{equation*}
Among the abundance of various RWPS, those which satisfy the `nonnegative linearization of products' property
\begin{equation}\label{eq:productlinear}
P_m(x)P_n(x)=\sum_{k=0}^{m+n}\underbrace{g(m,n;k)}_{\overset{!}{\geq}0}P_k(x)\;(m,n\in\mathbb{N}_0)
\end{equation}
are of particular interest because they are accompanied by a rich and sophisticated harmonic analysis. This crossing point between the theory of orthogonal polynomials and special functions, on the one hand, and functional and harmonic analysis, on the other hand, is fruitful, vivid and accompanied by recent publications \cite{FGS22,KS25,LO25,Vo23,Ya26}. Note that, as a consequence of orthogonality, the linearization coefficients in \eqref{eq:productlinear} always fulfill $g(m,n;|m-n|),g(m,n;m+n)\neq0$ and $g(m,n;k)=0$ for $k<|m-n|$ \cite{La05}. Moreover, it is an obvious consequence of symmetry that $g(m,n;k)=0$ if $m+n-k$ is odd. We particularly mention \cite[Theorem 2.2, Theorem 2.3]{LO25} concerning properties of the linearization coefficients. Again, there is an abundance of examples of RWPS which additionally satisfy nonnegative linearization of products. Nevertheless, determining the signs of the linearization coefficients $g(m,n;k)$ is generally a very nontrivial problem. Typically, the linearization coefficients $g(m,n;k)$ cannot be computed in a sufficiently explicit way. However, the literature contains some sufficient criteria for nonnegative linearization of products, and we will benefit from results of M\l{}otkowski and Szwarc \cite{MS01} in this context. In the literature, the nonnegative linearization of products property is also called `property (P)'.\\

In this paper, we consider an associated RWPS $(\widetilde{P_n}(x))_{n\in\mathbb{N}_0}$ which is given by ``switching'' the roles of $(a_n)_{n\in\mathbb{N}}$ and $(c_n)_{n\in\mathbb{N}}$ in \eqref{eq:threetermrec}, so $(\widetilde{P_n}(x))_{n\in\mathbb{N}_0}$ satisfies the recurrence relation $\widetilde{P_0}(x)=1$, $\widetilde{P_1}(x)=x$,
\begin{equation}\label{eq:threetermrecswitched}
x\widetilde{P_n}(x)=c_n\widetilde{P_{n+1}}(x)+a_n\widetilde{P_{n-1}}(x)\;(n\in\mathbb{N}).
\end{equation}
We are interested in the question when both $(P_n(x))_{n\in\mathbb{N}_0}$ and $(\widetilde{P_n}(x))_{n\in\mathbb{N}_0}$ satisfy nonnegative linearization of products, which was priorly considered by Lasser and Obermaier in the recent publication \cite{LO25}. It is clear that there is a trivial example, namely $a_n\equiv c_n\equiv1/2$, which corresponds to the Chebyshev polynomials of the first kind $(T_n(x))_{n\in\mathbb{N}_0}$ and the orthogonalization measure $1/\pi\cdot(1-x^2)^{-1/2}\chi_{(-1,1)}(x)\,\mathrm{d}x$ (one has $T_m(x)T_n(x)=T_{|m-n|}(x)/2+T_{m+n}(x)/2$ \cite{La83,La05}). The question whether there are further examples such that both $(P_n(x))_{n\in\mathbb{N}_0}$ and $(\widetilde{P_n}(x))_{n\in\mathbb{N}_0}$ satisfy nonnegative linearization of products or whether $(T_n(x))_{n\in\mathbb{N}_0}$ is characterized by this remarkable property was left open in \cite{LO25}.\\

Let us first briefly recall the above-mentioned harmonic analysis in some more detail \cite{La83,La05}. If $(P_n(x))_{n\in\mathbb{N}_0}$ satisfies nonnegative linearization of products, then the convolution $(m,n)\mapsto\sum_{k=|m-n|}^{m+n}g(m,n;k)\delta_k$ and the identity as involution make $\mathbb{N}_0$ a commutative discrete hypergroup with unit element $0$, a `polynomial hypergroup'.\footnote{We refer to standard literature like \cite{BH95,Je75,La23} concerning the precise definition of hypergroups in the sense of Jewett. The general axioms considerably simplify for discrete hypergroups \cite{La05,La23}. Roughly speaking, the difference between a locally compact group and a hypergroup is that in the latter case the convolution of two Dirac measures does not have to be a Dirac measure again but is allowed to be a more general probability measure (with certain non-degeneracy and compatibility properties).} The Haar function $h:\mathbb{N}_0\rightarrow[1,\infty)$,
\begin{equation}\label{eq:hdef}
h(n):=\frac{1}{\int_\mathbb{R}\!P_n^2(x)\,\mathrm{d}\mu(x)}=\frac{1}{g(n,n;0)}=\begin{cases} 1, & n=0, \\ \frac{1}{c_1}, & n=1, \\ \prod_{k=1}^n\frac{a_{k-1}}{c_k}, & n\geq2 \end{cases}
\end{equation}
has the following property: if, for any function $f\in\ell^1(h):=\{f:\mathbb{N}_0\rightarrow\mathbb{C}:\sum_{k=0}^\infty|f(k)|h(k)<\infty\}$ and any $n\in\mathbb{N}_0$, one considers the translation $T_n f:\mathbb{N}_0\rightarrow\mathbb{C}$, $T_n f(m):=\sum_{k=|m-n|}^{m+n}g(m,n;k)f(k)$, then $T_n f\in\ell^1(h)$ and $\sum_{k=0}^{\infty}T_n f(k)h(k)=\sum_{k=0}^{\infty}f(k)h(k)$. Endowed with the norm $\left\|f\right\|_1:=\sum_{k=0}^\infty|f(k)|h(k)$, with the convolution $f\ast g(n):=\sum_{k=0}^\infty T_n f(k)g(k)h(k)$ and with complex conjugation as involution, $\ell^1(h)$ becomes a semisimple commutative unital Banach $\ast$-algebra. The structure space $\Delta(\ell^1(h))$ can be identified with the compact set
\begin{equation*}
\mathcal{X}^b(\mathbb{N}_0):=\left\{z\in\mathbb{C}:\sup_{n\in\mathbb{N}_0}|P_n(z)|<\infty\right\}=\left\{z\in\mathbb{C}:\max_{n\in\mathbb{N}_0}|P_n(z)|=1\right\}
\end{equation*}
via $\mathcal{X}^b(\mathbb{N}_0)\rightarrow\Delta(\ell^1(h))$, $z\mapsto\varphi_z$ with $\varphi_z(f):=\sum_{k=0}^\infty f(k)\overline{P_k(z)}h(k)$, and in the same way the Hermitian structure space $\Delta_s(\ell^1(h))$ can be identified with
\begin{equation*}
\widehat{\mathbb{N}_0}:=\mathcal{X}^b(\mathbb{N}_0)\cap\mathbb{R}.
\end{equation*}
The orthogonalization measure $\mu$ can be regarded as Plancherel measure, and one always has $\{\pm1\}\cup\mathrm{supp}\;\mu\subseteq\widehat{\mathbb{N}_0}\subseteq[-1,1]$. However, in contrast to the group case the three dual objects $\mathrm{supp}\;\mu$, $\widehat{\mathbb{N}_0}$ and $\mathcal{X}^b(\mathbb{N}_0)$ can differ. We will also use that $\mathrm{supp}\;\mu=\sigma(T_1)$, where $T_1:\ell^2(h)\rightarrow\ell^2(h)$, $f\mapsto T_1f$ is regarded as a (bounded linear) operator on $\ell^2(h):=\{f:\mathbb{N}_0\rightarrow\mathbb{C}:\left\|f\right\|_2<\infty\}$ with $\left\|f\right\|_2:=\sqrt{\sum_{k=0}^\infty|f(k)|^2h(k)}$.\footnote{$T_1f$ for $f\in\ell^2(h)$ is defined like for $f\in\ell^1(h)$ above; in fact, the definition makes sense for every $f:\mathbb{N}_0\rightarrow\mathbb{C}$.}\\

We now come back to the central problem of this paper, namely whether there are nontrivial examples such that both $(P_n(x))_{n\in\mathbb{N}_0}$ and $(\widetilde{P_n}(x))_{n\in\mathbb{N}_0}$ satisfy nonnegative linearization of products. In the following, we use appropriate notation and write $\widetilde{h}$, $\widetilde{\widehat{\mathbb{N}_0}}$ and so on. The following result of Lasser and Obermaier, which is \cite[Theorem 4.3, Theorem 4.4]{LO25} and based on an analysis of the coefficients $g(m,n;|m-n|+2)$ and $g(m,n;m+n-2)$, provides a condition which potential examples must satisfy:

\begin{theorem}\label{thm:lasserobermaier}
If both $(P_n(x))_{n\in\mathbb{N}_0}$ and $(\widetilde{P_n}(x))_{n\in\mathbb{N}_0}$ satisfy nonnegative linearization of products and $(P_n(x))_{n\in\mathbb{N}_0}\neq(T_n(x))_{n\in\mathbb{N}_0}$, then either $c_{2n+1}>c_1>1/2$ and $c_{2n}<1/2$ for all $n\in\mathbb{N}$ or $c_{2n+1}<c_1<1/2$ and $c_{2n}>1/2$ for all $n\in\mathbb{N}$. In particular, $(c_n)_{n\in\mathbb{N}}$ cannot converge in this case.
\end{theorem}

Based on results of \cite{KS25}, one can give an example $(P_n(x))_{n\in\mathbb{N}_0}\neq(T_n(x))_{n\in\mathbb{N}_0}$ such that $(P_n(x))_{n\in\mathbb{N}_0}$ satisfies nonnegative linearization of products and the preceding necessary criterion for $(\widetilde{P_n}(x))_{n\in\mathbb{N}_0}$ satisfying nonnegative linearization of products, too, is fulfilled: an explicit possibility is to set $c_{2n-1}=(6n+4)/(9n+9)$ and $c_{2n}=(n+1)/(3n+5)$, which can be seen from \cite[Section 3]{KS25} (cf. also Theorem~\ref{thm:kahlerszwarc1} below for $\alpha=2$ and $\beta=5$, and cf. also Theorem 3.1 in the arXiv version arXiv:2212.11229v3 [math.CA] of \cite{KS25}). However, for this example $(\widetilde{P_n}(x))_{n\in\mathbb{N}_0}$ does not satisfy nonnegative linearization of products: in fact, a short computation shows that $\widetilde{g}(3,3;4)=-128/135$. This paper is devoted to the construction of examples such that $(\widetilde{P_n}(x))_{n\in\mathbb{N}_0}$ satisfies nonnegative linearization of products, too. In fact, the authors first conjectured that such examples do not exist at all, which would indeed have led to a characterization of the Chebyshev polynomials of the first kind. As another result of the paper, we shall find appropriate substitutes and characterize the Chebyshev polynomials of the first kind by additionally considering properties of the duals and (alternatively) Haar measures.

\subsection{Additional motivation and outline of the paper}

The problem under consideration can also be motivated as follows: in \cite{KS25}, Kahler and Szwarc gave some explanations for the observation that even some ``unexpected'' examples of polynomial hypergroups, namely polynomial hypergroups induced by certain generalized Chebyshev polynomials \cite{Ka21a}, have the property $h(n)\geq2\;(n\in\mathbb{N})$ \cite[Section 2]{KS25} (in general, it is only clear that $h(n)>1\;(n\in\mathbb{N})$). However, they also provided the presumably first examples of polynomial hypergroups which violate this property \cite[Section 3]{KS25}. We recall two results from \cite[Section 3]{KS25}, cf. particularly \cite[Theorem 3.1]{KS25} and \cite[Theorem 3.2]{KS25}, in detail:

\begin{theorem}\label{thm:kahlerszwarc1}
Let $\alpha,\beta\geq2$. Then $(c_n)_{n\in\mathbb{N}}$ can be chosen in such a way that
\begin{equation*}
\alpha_n=\begin{cases} \frac{\sqrt{\beta}}{\sqrt{\alpha-1}+\sqrt{\beta-1}}, & n=1, \\ \frac{\sqrt{\alpha-1}}{\sqrt{\alpha-1}+\sqrt{\beta-1}}, & n\;\mbox{even}, \\ \frac{\sqrt{\beta-1}}{\sqrt{\alpha-1}+\sqrt{\beta-1}}, & \mbox{else}, \end{cases}
\end{equation*}
and the corresponding RWPS $(P_n(x))_{n\in\mathbb{N}_0}$ satisfies nonnegative linearization of products with an increasing\footnote{We use the expression ``increasing'' in the sense of ``nondecreasing'', i.e., we do not require strict monotonicity unless explicitly stated otherwise.} Haar function $h$ with $h(2)\geq2$. However, for every $\epsilon\in(0,1)$ and any choice of $\alpha$ the parameter $\beta$ can be chosen in such a way that $h(1)=1+\epsilon$.
\end{theorem}

\begin{theorem}\label{thm:kahlerszwarc2}
Let $(s_n)_{n\in\mathbb{N}}\subseteq(0,1)$ be a convex null sequence. Then $(c_n)_{n\in\mathbb{N}}$ can be chosen in such a way that
\begin{equation*}
\alpha_n=\begin{cases} 1-s_{\frac{n+1}{2}}, & n\;\mbox{odd}, \\ s_{\frac{n+2}{2}}-s_{\frac{n+4}{2}}, & n\;\mbox{even}, \end{cases}
\end{equation*}
and the corresponding RWPS $(P_n(x))_{n\in\mathbb{N}_0}$ satisfies nonnegative linearization of products with a strictly increasing Haar function $h$ with $h(2)>4$. However, for every $\epsilon\in(0,1)$ the sequence $(s_n)_{n\in\mathbb{N}}$ can be chosen in such a way that $h(1)=1+\epsilon$.
\end{theorem}

The relation to the problem we study in this paper is as follows: if both $(P_n(x))_{n\in\mathbb{N}_0}$ and $(\widetilde{P_n}(x))_{n\in\mathbb{N}_0}$ satisfy nonnegative linearization of products and $(P_n(x))_{n\in\mathbb{N}_0}\neq(T_n(x))_{n\in\mathbb{N}_0}$, then, by \eqref{eq:hdef}, one has $h(1)=1/c_1$ and $\widetilde{h}(1)=1/(1-c_1)$. Since $c_1\neq1/2$ by Theorem~\ref{thm:lasserobermaier}, one has either $h(1)<2$ or $\widetilde{h}(1)<2$. In particular, any solution to the problem under consideration also provides examples such that the first Haar weight deceeds $2$. In other words: the existence of RWPS $(P_n(x))_{n\in\mathbb{N}_0}\neq(T_n(x))_{n\in\mathbb{N}_0}$ such that both $(P_n(x))_{n\in\mathbb{N}_0}$ and $(\widetilde{P_n}(x))_{n\in\mathbb{N}_0}$ satisfy nonnegative linearization of products is a sharpening of the result that RWPS with $h(1)<2$ exist.\\

Besides this motivation from the theory of orthogonal polynomials and special functions (and its intersections with functional and harmonic analysis), we want to mention two further motivations. The first one is of combinatorial type: nonnegative linearization of products can be characterized via certain discrete boundary value problems and the corresponding existence of matrices with specific properties \cite{Sz03,Sz05}, and the switching of $(a_n)_{n\in\mathbb{N}}$ and $(c_n)_{n\in\mathbb{N}}$ can be translated into such considerations. The second one is of probabilistic type: the recurrence coefficients $(a_n)_{n\in\mathbb{N}}$ and $(c_n)_{n\in\mathbb{N}}$ can be interpreted as $1$-step transition probabilities of certain corresponding random walks \cite{CSvD98,KM59,vDS93}, and the switching of $(a_n)_{n\in\mathbb{N}}$ and $(c_n)_{n\in\mathbb{N}}$ obviously just corresponds to opposite probabilities. If both $(P_n(x))_{n\in\mathbb{N}_0}$ and $(\widetilde{P_n}(x))_{n\in\mathbb{N}_0}$ satisfy nonnegative linearization of products, however, then the resulting tools of harmonic analysis are available for both versions.\\

The paper is organized as follows: in Section~\ref{sec:main}, we present our main results and particularly construct RWPS $(P_n(x))_{n\in\mathbb{N}_0}\neq(T_n(x))_{n\in\mathbb{N}_0}$ such that both $(P_n(x))_{n\in\mathbb{N}_0}$ and $(\widetilde{P_n}(x))_{n\in\mathbb{N}_0}$ satisfy nonnegative linearization of products. More precisely, we provide a sufficient criterion for the recurrence coefficients $(c_n)_{n\in\mathbb{N}}$ and give explicit examples which fulfill this criterion, and we also present the above-mentioned characterizations of $(T_n(x))_{n\in\mathbb{N}_0}$. To our knowledge, our criterion also yields the first examples of polynomial hypergroups whose Haar function is not increasing\footnote{In fact, we shall see that RWPS $(P_n(x))_{n\in\mathbb{N}_0}\neq(T_n(x))_{n\in\mathbb{N}_0}$ for which both $(P_n(x))_{n\in\mathbb{N}_0}$ and $(\widetilde{P_n}(x))_{n\in\mathbb{N}_0}$ satisfy nonnegative linearization of products can never be accompanied by two increasing Haar functions $h$ and $\widetilde{h}$, which can be seen as another motivation for the problem under consideration.}, as well as the first examples whose second Haar weights deceed $2$. This highlights the diversity of orthogonal polynomials and particularly answers some open problems mentioned in \cite{KS25}. In Section~\ref{sec:proofs}, we give the proofs. Finally, in Section~\ref{sec:further} we collect some more properties of the polynomials provided by our sufficient criterion. In particular, we are interested in properties of the corresponding orthogonalization measures.\\

We remark that our research benefited from computer algebra systems (Maple) to find suitable simplifications and estimations, get conjectures, do numerical experiments and so on. However, the final proofs can be understood without any computer usage.

\section{Statement of the main results}\label{sec:main}

The following result provides the above-mentioned sufficient criterion whose application will solve the problem under consideration.

\begin{theorem}\label{thm:sufficientcriterion}
Let $(s_n)_{n\in\mathbb{N}}\subseteq(0,1)$ be a sequence such that
\begin{equation}\label{eq:suf1}
s_2<1-\frac{s_1}{1-s_1}
\end{equation}
and
\begin{equation}\label{eq:suf2}
s_n\leq\frac{s_{n-2}}{2}-2s_{n-1}\;(n\geq3),
\end{equation}
and let $(c_n)_{n\in\mathbb{N}}$ be given by either
\begin{equation}\label{eq:cnfirst}
c_n:=\begin{cases} 1-s_n, & n\;\mbox{odd}, \\ s_n, & n\;\mbox{even} \end{cases}
\end{equation}
or
\begin{equation}\label{eq:cnsecond}
c_n:=\begin{cases} s_n, & n\;\mbox{odd}, \\ 1-s_n, & n\;\mbox{even}. \end{cases}
\end{equation}
Then both $(P_n(x))_{n\in\mathbb{N}_0}$ and $(\widetilde{P_n}(x))_{n\in\mathbb{N}_0}$ satisfy nonnegative linearization of products.
\end{theorem}

Some properties of the sequence $(s_n)_{n\in\mathbb{N}}$ above, particularly concerning growth and monotonicity, will be considered in Lemma~\ref{lma:sufficientcriterion}. Based on Theorem~\ref{thm:sufficientcriterion}, we immediately obtain the existence of RWPS $(P_n(x))_{n\in\mathbb{N}_0}\neq(T_n(x))_{n\in\mathbb{N}_0}$ such that both $(P_n(x))_{n\in\mathbb{N}_0}$ and $(\widetilde{P_n}(x))_{n\in\mathbb{N}_0}$ satisfy nonnegative linearization of products. For instance, the choices $s_n=1/5^n$ or $s_n=1/(n+3)!$ satisfy the conditions of Theorem~\ref{thm:sufficientcriterion}. However, we can show more and find a variety of such RWPS. Furthermore, we classify the RWPS $(P_n(x))_{n\in\mathbb{N}_0}$ such that both $(P_n(x))_{n\in\mathbb{N}_0}$ and $(\widetilde{P_n}(x))_{n\in\mathbb{N}_0}$ satisfy nonnegative linearization of products by properties of their duals. In particular, we can characterize $(T_n(x))_{n\in\mathbb{N}_0}$ via an additional property.

\begin{theorem}\label{thm:actualresult}
\begin{enumerate}[(i)]
\item For every $C\in(0,1)\backslash\{1/2\}$, there exists an RWPS $(P_n(x))_{n\in\mathbb{N}_0}\neq(T_n(x))_{n\in\mathbb{N}_0}$ with $c_1=C$ such that both $(P_n(x))_{n\in\mathbb{N}_0}$ and $(\widetilde{P_n}(x))_{n\in\mathbb{N}_0}$ satisfy nonnegative linearization of products.
\item $(P_n(x))_{n\in\mathbb{N}_0}$ and $(\widetilde{P_n}(x))_{n\in\mathbb{N}_0}$ satisfy nonnegative linearization of products and $0\in\widehat{\mathbb{N}_0}\cap\widetilde{\widehat{\mathbb{N}_0}}$ if and only if $(P_n(x))_{n\in\mathbb{N}_0}=(T_n(x))_{n\in\mathbb{N}_0}$.
\item If $(P_n(x))_{n\in\mathbb{N}_0}$ and $(\widetilde{P_n}(x))_{n\in\mathbb{N}_0}$ satisfy nonnegative linearization of products, then $0\in\widehat{\mathbb{N}_0}$ (if $c_1\leq1/2$) or $0\in\widetilde{\widehat{\mathbb{N}_0}}$ (if $c_1\geq1/2$).
\end{enumerate}
\end{theorem}

We mention with regard to Theorem~\ref{thm:actualresult} (ii) that there are previous, related characterizations by Lasser and Obermaier: it obviously follows from Theorem~\ref{thm:lasserobermaier} that if both $(P_n(x))_{n\in\mathbb{N}_0}$ and $(\widetilde{P_n}(x))_{n\in\mathbb{N}_0}$ satisfy nonnegative linearization of products, and if one additionally requires that $c_n=1/2$ for some $n\in\mathbb{N}$ or that $(c_n)_{n\in\mathbb{N}}$ converges, then $(P_n(x))_{n\in\mathbb{N}_0}=(T_n(x))_{n\in\mathbb{N}_0}$ (cf. \cite[Theorem 4.2, Theorem 4.4]{LO25}). Note that the second part of Theorem~\ref{thm:lasserobermaier} can easily be reobtained from Theorem~\ref{thm:actualresult} (ii): if both $(P_n(x))_{n\in\mathbb{N}_0}$ and $(\widetilde{P_n}(x))_{n\in\mathbb{N}_0}$ satisfy nonnegative linearization of products and $(c_n)_{n\in\mathbb{N}}$ converges, then $\widehat{\mathbb{N}_0}=\widetilde{\widehat{\mathbb{N}_0}}=[-1,1]$ \cite[Theorem (2.2)]{La94}. Moreover, note that, given an arbitrary $\epsilon\in(0,1)$, the existence of RWPS $(P_n(x))_{n\in\mathbb{N}_0}$ such that nonnegative linearization of products is satisfied and $h(1)=1+\epsilon$ (cf. Theorem~\ref{thm:kahlerszwarc1} and Theorem~\ref{thm:kahlerszwarc2}) can be reobtained from Theorem~\ref{thm:actualresult} (i) because $h(1)=1/c_1=1/C$.\\

We finally obtain two results with regard to the Haar measures, which, to our knowledge, are in sharp contrast to all previously known examples of polynomial hypergroups (cf. also Theorem~\ref{thm:kahlerszwarc1} and Theorem~\ref{thm:kahlerszwarc2} again). The first one contains another characterization of $(T_n(x))_{n\in\mathbb{N}_0}$.

\begin{theorem}\label{thm:haarnotnondecreasing}
There exist RWPS $(P_n(x))_{n\in\mathbb{N}_0}$ such that nonnegative linearization of products is satisfied and the Haar function $h$ is not increasing. More precisely, one has the following:
\begin{enumerate}[(i)]
\item $(P_n(x))_{n\in\mathbb{N}_0}$ and $(\widetilde{P_n}(x))_{n\in\mathbb{N}_0}$ satisfy nonnegative linearization of products and both $h$ and $\widetilde{h}$ are increasing if and only if $(P_n(x))_{n\in\mathbb{N}_0}=(T_n(x))_{n\in\mathbb{N}_0}$.
\item If $(P_n(x))_{n\in\mathbb{N}_0}$ is as in Theorem~\ref{thm:sufficientcriterion}, then $h(2n)<h(2n-1)$ for all $n\in\mathbb{N}$ or $h(2n+1)<h(2n)$ for all $n\in\mathbb{N}$.
\end{enumerate} 
\end{theorem}

\begin{theorem}\label{thm:haar2smaller2}
For every $\epsilon\in(0,1)$, there exist RWPS $(P_n(x))_{n\in\mathbb{N}_0}$ such that nonnegative linearization of products is satisfied and the Haar function $h$ fulfills $h(2)=1+\epsilon$.
\end{theorem}

Theorem~\ref{thm:haar2smaller2} gives full answers to the open problems (i) and (iii) presented in \cite[Section 4]{KS25}, and it can be seen from the proof below that one also obtains partial answers to the open problems (iv) and (v) of \cite[Section 4]{KS25}. In particularly, we shall find examples such that $h(1)>2$, $h(2)<2$ and $0\in\widehat{\mathbb{N}_0}$.

\section{Proof of the main results}\label{sec:proofs}

The proof of Theorem~\ref{thm:sufficientcriterion} relies on two essential ingredients. The first ingredient is the following sufficient criterion for nonnegative linearization of products \cite[Theorem 3]{MS01}:

\begin{theorem}\label{thm:ms01}
Let $(P_n(x))_{n\in\mathbb{N}_0}$ be an RWPS.
\begin{enumerate}[(i)]
\item If the tridiagonal matrix
\begin{equation*}
\begin{pmatrix} \alpha_{2N} & \alpha_1 & \vphantom{ddots} & \phantom{ddots} & (0) \\ \alpha_1 & \alpha_{2N} & \alpha_2 & \phantom{ddots} & \phantom{ddots} \\ \phantom{ddots} & \alpha_2 & \ddots & \ddots & \phantom{ddots} \\ \phantom{ddots} & \phantom{ddots} & \ddots & \phantom{\ddots} & \alpha_{2N-1} \\ (0) & \phantom{ddots} & \phantom{ddots} & \alpha_{2N-1} & \alpha_{2N} \end{pmatrix}
\end{equation*}
is positive definite for all $N\in\mathbb{N}$, then $(P_n(x))_{n\in\mathbb{N}_0}$ satisfies nonnegative linearization of products.
\item If the tridiagonal matrix
\begin{equation*}
\begin{pmatrix} \alpha_{2N+1} & \alpha_1 & \vphantom{ddots} & \phantom{ddots} & (0) \\ \alpha_1 & \alpha_{2N+1} & \alpha_2 & \phantom{ddots} & \phantom{ddots} \\ \phantom{ddots} & \alpha_2 & \ddots & \ddots & \phantom{ddots} \\ \phantom{ddots} & \phantom{ddots} & \ddots & \phantom{\ddots} & \alpha_{2N} \\ (0) & \phantom{ddots} & \phantom{ddots} & \alpha_{2N} & \alpha_{2N+1} \end{pmatrix}
\end{equation*}
is positive definite for all $N\in\mathbb{N}$, then $(P_n(x))_{n\in\mathbb{N}_0}$ satisfies nonnegative linearization of products.
\end{enumerate}
\end{theorem}

It is well-known that strictly diagonally dominant symmetric matrices with positive diagonal elements are positive definite. However, as a second ingredient we need a more sophisticated criterion for positive definiteness \cite[Theorem 2.3]{El03}:

\begin{theorem}\label{thm:el03}
For every $N\in\mathbb{N}$, a (real) symmetric tridiagonal matrix of the form
\begin{equation*}
\begin{pmatrix} D_1 & A_1 & \vphantom{ddots} & \phantom{ddots} & (0) \\ A_1 & D_2 & A_2 & \phantom{ddots} & \phantom{ddots} \\ \phantom{ddots} & A_2 & \ddots & \ddots & \phantom{ddots} \\ \phantom{ddots} & \phantom{ddots} & \ddots & \phantom{\ddots} & A_{N-1} \\ (0) & \phantom{ddots} & \phantom{ddots} & A_{N-1} & D_N \end{pmatrix}
\end{equation*}
is positive definite if and only if there exist $t_1,\ldots,t_N>0$ with $t_1=D_1$ and $t_{n+1}=D_{n+1}-A_n^2/t_n$ for all $n\in\{1,\ldots,N-1\}$.
\end{theorem}

Before coming to the proof of Theorem~\ref{thm:sufficientcriterion}, we collect some properties of the sequence $(s_n)_{n\in\mathbb{N}}$ and some further preparations.

\begin{lemma}\label{lma:sufficientcriterion}
Let $(s_n)_{n\in\mathbb{N}}$ be as in Theorem~\ref{thm:sufficientcriterion}. Then $(s_n)_{n\in\mathbb{N}}$ is strictly decreasing and exponentially tends to $0$. More precisely, one has
\begin{equation}\label{eq:lma1}
s_n<\frac{s_1}{4^{n-1}}\;(n\geq2)
\end{equation}
with $s_1<1/2$. Moreover, one has
\begin{equation}\label{eq:sufadd1}
s_3<\frac{s_1-s_2}{1-s_2}-\frac{s_2}{(1-s_2)^2}
\end{equation}
and
\begin{equation}\label{eq:sufadd2}
s_4<1-\frac{s_1s_2}{1-s_3}-\frac{1-s_2+s_2s_3}{(1-s_3)^2}.
\end{equation}
\end{lemma}

\begin{proof}
The assertion $s_1<1/2$ follows immediately from \eqref{eq:suf1}, which yields
\begin{equation*}
0<1-\frac{s_1}{1-s_1}.
\end{equation*}
Moreover, as a consequence of \eqref{eq:suf2} we have $s_{n+1}<s_n/4$ for all $n\in\mathbb{N}$, which shows the monotonicity behavior and \eqref{eq:lma1}. \eqref{eq:sufadd1} is another consequence of \eqref{eq:suf2}, which can be seen as follows: it is sufficient to prove that
\begin{equation*}
\frac{s_1}{2}-2s_2<\frac{s_1-s_2}{1-s_2}-\frac{s_2}{(1-s_2)^2}
\end{equation*}
or, equivalently,
\begin{equation*}
6s_2^2-4s_2^3<s_1(1-s_2^2).
\end{equation*}
Since $s_1>4s_2$ by \eqref{eq:lma1}, it therefore suffices to show that
\begin{equation*}
6s_2^2-4s_2^3<4s_2\cdot(1-s_2^2).
\end{equation*}
The latter, however, reduces to the true assertion $s_2<2/3$. To conclude \eqref{eq:sufadd2} from \eqref{eq:suf2}, we show that
\begin{equation*}
\frac{s_2}{2}-2s_3<1-\frac{s_1s_2}{1-s_3}-\frac{1-s_2+s_2s_3}{(1-s_3)^2}
\end{equation*}
or, equivalently,
\begin{equation*}
2s_1s_2(1-s_3)<s_2(1-s_3^2)+4s_3^3-6s_3^2.
\end{equation*}
Since, by \eqref{eq:suf1}, $s_1<(1-s_2)/(2-s_2)$, it suffices to prove that
\begin{equation*}
\frac{2-2s_2}{2-s_2}\cdot s_2(1-s_3)<s_2(1-s_3^2)+4s_3^3-6s_3^2
\end{equation*}
or, equivalently,
\begin{equation*}
(1-s_3)^2s_2^2+(2s_3+4s_3^2-4s_3^3)s_2+8s_3^3-12s_3^2>0.
\end{equation*}
The latter can be seen as follows: let $f:\mathbb{R}\rightarrow\mathbb{R}$ be defined by $f(x):=(1-s_3)^2x^2+(2s_3+4s_3^2-4s_3^3)x+8s_3^3-12s_3^2$. Then
\begin{equation*}
f^{\prime}(4s_3)=10s_3-12s_3^2+4s_3^3>0
\end{equation*}
and
\begin{equation*}
f^{\prime}(1)=2-2s_3+6s_3^2-4s_3^3>0,
\end{equation*}
so $f$ is strictly increasing on $(4s_3,1)$. Since $s_2>4s_3$, this yields
\begin{align*}
&(1-s_3)^2s_2^2+(2s_3+4s_3^2-4s_3^3)s_2+8s_3^3-12s_3^2\\
&>(1-s_3)^2(4s_3)^2+(2s_3+4s_3^2-4s_3^3)\cdot4s_3+8s_3^3-12s_3^2\\
&=12s_3^2-8s_3^3\\
&>0.
\end{align*}
\end{proof}

\begin{lemma}\label{lma:sufficientcriterion2}
Let $(s_n)_{n\in\mathbb{N}}$ be as in Theorem~\ref{thm:sufficientcriterion}, and let $(c_n)_{n\in\mathbb{N}}$ be as in \eqref{eq:cnfirst}. Then $(\alpha_{2n-1}^2)_{n\in\mathbb{N}}$ and $(\widetilde{\alpha_{2n}}^2)_{n\in\mathbb{N}}$ are strictly increasing.
\end{lemma}

\begin{proof}
The asserted monotonicity behavior is a consequence of Lemma~\ref{lma:sufficientcriterion}: on the one hand, by \eqref{eq:sufadd1} we have
\begin{equation*}
s_3<\frac{s_1-s_2}{1-s_2}
\end{equation*}
and consequently
\begin{equation*}
\alpha_3^2-\alpha_1^2=c_3a_2-c_1=(1-s_3)(1-s_2)-(1-s_1)>0.
\end{equation*}
On the other hand, it is clear from the monotonicity of $(s_n)_{n\in\mathbb{N}}$ that
\begin{equation*}
\alpha_{2n+1}^2-\alpha_{2n-1}^2=c_{2n+1}a_{2n}-c_{2n-1}a_{2n-2}=(1-s_{2n+1})(1-s_{2n})-(1-s_{2n-1})(1-s_{2n-2})>0
\end{equation*}
for all $n\geq2$. Hence, we obtain that $(\alpha_{2n-1}^2)_{n\in\mathbb{N}}$ is strictly increasing. In the same way, one gets that $(\widetilde{\alpha_{2n}}^2)_{n\in\mathbb{N}}$ is strictly increasing.
\end{proof}

\begin{lemma}\label{lma:sufficientcriterion3}
Let $(s_n)_{n\in\mathbb{N}}$ be as in Theorem~\ref{thm:sufficientcriterion}, and let $(c_n)_{n\in\mathbb{N}}$ be as in \eqref{eq:cnfirst}. If $N\geq2$, then
\begin{equation}\label{eq:sufproof1P}
(1-\alpha_{2N+1}^2a_{2n+2})\alpha_{2n+1}^2<\alpha_{2N+1}^2a_{2n}(\alpha_{2N+1}^2-\alpha_{2N+1}^4a_{2n+2}-\alpha_{2n+2}^2)
\end{equation}
and
\begin{equation}\label{eq:sufproof2P}
\alpha_{2n+1}^2<\alpha_{2N+1}^4a_{2n}
\end{equation}
for all $n\in\{1,\ldots,N-1\}$. Furthermore, if $N\geq3$, then
\begin{equation}\label{eq:sufproof1Q}
(1-\widetilde{\alpha_{2N}}^2c_{2n+3})\widetilde{\alpha_{2n+2}}^2<\widetilde{\alpha_{2N}}^2c_{2n+1}(\widetilde{\alpha_{2N}}^2-\widetilde{\alpha_{2N}}^4c_{2n+3}-\widetilde{\alpha_{2n+3}}^2).
\end{equation}
and
\begin{equation}\label{eq:sufproof2Q}
\widetilde{\alpha_{2n+2}}^2<\widetilde{\alpha_{2N}}^4c_{2n+1}
\end{equation}
for all $n\in\{1,\ldots,N-2\}$.
\end{lemma}

\begin{proof}
Let $N\geq2$ and $n\in\{1,\ldots,N-1\}$. By \eqref{eq:suf2}, one has
\begin{equation*}
2c_{2k+4}+c_{2k+3}-(2c_{2k+2}+c_{2k+1})=2s_{2k+4}+s_{2k+1}-2s_{2k+2}-s_{2k+3}>0
\end{equation*}
for all $k\in\mathbb{N}$. Hence, one has
\begin{equation*}
2c_{2n+2}+c_{2n+1}\leq2c_{2N}+c_{2N-1}.
\end{equation*}
Moreover, we can estimate
\begin{align*}
c_{2N+1}^2a_{2N}^2-(2c_{2N}+c_{2N-1})&=(1-s_{2N+1})^2(1-s_{2N})^2-2s_{2N}-1+s_{2N-1}\\
&>(1-2s_{2N+1})(1-2s_{2N})-2s_{2N}-1+s_{2N-1}\\
&>-2s_{2N}-2s_{2N+1}-2s_{2N}+s_{2N-1}\\
&=s_{2N-1}-4s_{2N}-2s_{2N+1}\\
&\geq0,
\end{align*}
where we have used condition \eqref{eq:suf2} again. Putting both together, we obtain
\begin{equation}\label{eq:sufproofcentralP}
2c_{2n+2}+c_{2n+1}<c_{2N+1}^2a_{2N}^2
\end{equation}
and therefore
\begin{equation}\label{eq:sufproof3P}
(c_{2N+1}a_{2N}+c_{2N+1}^2a_{2N}^2)c_{2n+2}+c_{2n+1}<c_{2N+1}^2a_{2N}^2.
\end{equation}
Multiplying the latter inequality with $(1-c_{2N+1}a_{2N})a_{2n}$, we obtain \eqref{eq:sufproof1P} after a short calculation (take into account that the recurrence coefficients $a_n$ and $c_n$ always sum up to $1$). Moreover, as another consequence of \eqref{eq:sufproofcentralP}, we obtain
\begin{equation*}
c_{2N+1}^2a_{2N}^2>c_{2n+1}.
\end{equation*}
Multiplication with $a_{2n}$ yields \eqref{eq:sufproof2P}. Now let $N\geq3$ and $n\in\{1,\ldots,N-2\}$. We have
\begin{equation}\label{eq:sufproofcentralQ}
2a_{2n+3}+a_{2n+2}<a_{2N}^2c_{2N-1}^2
\end{equation}
and
\begin{equation}\label{eq:sufproof3Q}
(a_{2N}c_{2N-1}+a_{2N}^2c_{2N-1}^2)a_{2n+3}+a_{2n+2}<a_{2N}^2c_{2N-1}^2
\end{equation}
as analogues to \eqref{eq:sufproofcentralP} and \eqref{eq:sufproof3P}. After multiplication with $(1-a_{2N}c_{2N-1})c_{2n+1}$, \eqref{eq:sufproof3Q} yields \eqref{eq:sufproof1Q}. Moreover, \eqref{eq:sufproofcentralQ} also yields
\begin{equation*}
a_{2N}^2c_{2N-1}^2>a_{2n+2}.
\end{equation*}
If we multiply the latter with $c_{2n+1}$, we obtain \eqref{eq:sufproof2Q}.
\end{proof}

\begin{proof}[Proof of Theorem~\ref{thm:sufficientcriterion}]
It is clear that it suffices to assume that $(c_n)_{n\in\mathbb{N}}$ is as in \eqref{eq:cnfirst}, which shall be assumed from now on, because if $(c_n)_{n\in\mathbb{N}}$ is as in \eqref{eq:cnsecond}, then the roles of $(P_n(x))_{n\in\mathbb{N}_0}$ and $(\widetilde{P_n}(x))_{n\in\mathbb{N}_0}$ are just changed. Our strategy is to show that $(P_n(x))_{n\in\mathbb{N}_0}$ fits in Theorem~\ref{thm:ms01} (ii), whereas $(\widetilde{P_n}(x))_{n\in\mathbb{N}_0}$ fits in Theorem~\ref{thm:ms01} (i) (the latter with $\alpha_n$ replaced by $\widetilde{\alpha_n}$), and Theorem~\ref{thm:el03} will be the tool to establish the positive definiteness of the occurring matrices. Concerning an application of Theorem~\ref{thm:el03}, one of the key ideas will be to appropriately separate the construction of $t_1,t_3,\ldots$ from the construction of $t_2,t_4,\ldots$.\\

We first show that $(P_n(x))_{n\in\mathbb{N}_0}$ satisfies nonnegative linearization of products and divide the proof into two steps. Let $N\in\mathbb{N}$ be arbitrary but fixed.\\

\textit{Step 1:} we show that there exist $t_1,t_3,\ldots,t_{2N+1}\in\mathbb{R}$ with $t_1=\alpha_{2N+1}$ and
\begin{equation}\label{eq:sufproof4P}
t_{2n+1}=\alpha_{2N+1}-\frac{\alpha_{2n}^2t_{2n-1}}{\alpha_{2N+1}t_{2n-1}-\alpha_{2n-1}^2}
\end{equation}
and
\begin{equation}\label{eq:sufproof5P}
\alpha_{2N+1}^3a_{2n}<t_{2n+1}
\end{equation}
for all $n\in\{1,\ldots,N\}$ (the estimation provided by \eqref{eq:sufproof5P} is the essential auxiliary tool that makes the proof work, cf. also Remark~\ref{rem:sufficientcriterion}). We start with $t_1:=\alpha_{2N+1}$ and $t_3:=\alpha_{2N+1}-\alpha_2^2t_1/(\alpha_{2N+1}t_1-\alpha_1^2)$. Since $\alpha_{2N+1}t_1-\alpha_1^2=\alpha_{2N+1}^2-\alpha_1^2>0$ (Lemma~\ref{lma:sufficientcriterion2}), $t_3$ is well-defined. Moreover, we have
\begin{equation*}
t_3-\alpha_{2N+1}^3a_2=\frac{\alpha_{2N+1}(1-\alpha_{2N+1}^2)(\alpha_{2N+1}^2a_2-c_1-c_2a_1)}{\alpha_{2N+1}^2-\alpha_1^2}
\end{equation*}
(short calculation which again uses that the recurrence coefficients $a_n$ and $c_n$ always sum up to $1$, which will be used without further explanation in the following) and
\begin{align*}
\alpha_{2N+1}^2a_2-c_1-c_2a_1&\geq\alpha_3^2a_2-c_1-c_2a_1\\
&=c_3a_2^2-c_1-c_2a_1\\
&=(1-s_3)(1-s_2)^2-(1-s_1)-s_2s_1\\
&>0
\end{align*}
by \eqref{eq:sufadd1} and Lemma~\ref{lma:sufficientcriterion2}. This shows \eqref{eq:sufproof4P} and \eqref{eq:sufproof5P} for $n=1$. Now let $n\in\{1,\ldots,N-1\}$ be arbitrary but fixed, and assume that $t_{2n-1}$ and $t_{2n+1}$ are such that \eqref{eq:sufproof4P} and \eqref{eq:sufproof5P} are satisfied for $n$. We then define $t_{2n+3}:=\alpha_{2N+1}-\alpha_{2n+2}^2t_{2n+1}/(\alpha_{2N+1}t_{2n+1}-\alpha_{2n+1}^2)$. The well-definedness is a consequence of the assumption that \eqref{eq:sufproof5P} is satisfied for $n$ combined with \eqref{eq:sufproof2P} (Lemma~\ref{lma:sufficientcriterion3}): in fact, we can estimate
\begin{equation*}
\alpha_{2N+1}t_{2n+1}-\alpha_{2n+1}^2>\alpha_{2N+1}^4a_{2n}-\alpha_{2n+1}^2
\end{equation*}
and therefore
\begin{equation}\label{eq:sufproofadd1P}
\alpha_{2N+1}t_{2n+1}-\alpha_{2n+1}^2>0.
\end{equation}
By construction, it is clear that \eqref{eq:sufproof4P} is valid for $n+1$. It remains to show that \eqref{eq:sufproof5P} holds true for $n+1$, too, which can be seen from \eqref{eq:sufproof1P} as follows: by \eqref{eq:sufproof1P}, we know that
\begin{equation}\label{eq:sufproofadd2P}
\alpha_{2N+1}^2-\alpha_{2N+1}^4a_{2n+2}-\alpha_{2n+2}^2>0
\end{equation}
and
\begin{equation*}
\frac{\alpha_{2N+1}(1-\alpha_{2N+1}^2a_{2n+2})\alpha_{2n+1}^2}{\alpha_{2N+1}^2-\alpha_{2N+1}^4a_{2n+2}-\alpha_{2n+2}^2}<\alpha_{2N+1}^3a_{2n},
\end{equation*}
so
\begin{equation*}
\frac{\alpha_{2N+1}(1-\alpha_{2N+1}^2a_{2n+2})\alpha_{2n+1}^2}{\alpha_{2N+1}^2-\alpha_{2N+1}^4a_{2n+2}-\alpha_{2n+2}^2}<t_{2n+1}.
\end{equation*}
Taking into account \eqref{eq:sufproofadd1P} and \eqref{eq:sufproofadd2P} again, we can reformulate the latter as
\begin{equation*}
\alpha_{2N+1}^3a_{2n+2}<\alpha_{2N+1}-\frac{\alpha_{2n+2}^2t_{2n+1}}{\alpha_{2N+1}t_{2n+1}-\alpha_{2n+1}^2}=t_{2n+3}.
\end{equation*}
Before proceeding with the second step, we emphasize that our construction particularly implies that
\begin{equation}\label{eq:fromstep2P}
t_3,t_5,\ldots,t_{2N+1}\in(0,\alpha_{2N+1}).
\end{equation}

\textit{Step 2:} as a consequence of \eqref{eq:fromstep2P}, we have
\begin{equation*}
t_{2n}:=\frac{\alpha_{2n}^2}{\alpha_{2N+1}-t_{2n+1}}>0
\end{equation*}
for all $n\in\{1,\ldots,N\}$. Moreover, for all $n\in\{1,\ldots,N\}$ this construction and \eqref{eq:sufproof4P} yield
\begin{equation*}
t_{2n+1}=\alpha_{2N+1}-\frac{\alpha_{2n}^2}{t_{2n}}
\end{equation*}
and
\begin{equation*}
t_{2n}=\frac{\alpha_{2n}^2}{\alpha_{2N+1}-\left(\alpha_{2N+1}-\frac{\alpha_{2n}^2t_{2n-1}}{\alpha_{2N+1}t_{2n-1}-\alpha_{2n-1}^2}\right)}=\alpha_{2N+1}-\frac{\alpha_{2n-1}^2}{t_{2n-1}}.
\end{equation*}
In conclusion, we have shown that $t_1,\ldots,t_{2N+1}>0$ with $t_1=\alpha_{2N+1}$ and
\begin{equation*}
t_{n+1}=\alpha_{2N+1}-\frac{\alpha_n^2}{t_n}
\end{equation*}
for all $n\in\{1,\ldots,2N\}$. By Theorem~\ref{thm:el03}, this proves that the tridiagonal matrix
\begin{equation*}
\begin{pmatrix} \alpha_{2N+1} & \alpha_1 & \vphantom{ddots} & \phantom{ddots} & (0) \\ \alpha_1 & \alpha_{2N+1} & \alpha_2 & \phantom{ddots} & \phantom{ddots} \\ \phantom{ddots} & \alpha_2 & \ddots & \ddots & \phantom{ddots} \\ \phantom{ddots} & \phantom{ddots} & \ddots & \phantom{\ddots} & \alpha_{2N} \\ (0) & \phantom{ddots} & \phantom{ddots} & \alpha_{2N} & \alpha_{2N+1} \end{pmatrix}
\end{equation*}
is positive definite. Since $N$ was arbitrary, it follows from Theorem~\ref{thm:ms01} (ii) that $(P_n(x))_{n\in\mathbb{N}_0}$ satisfies nonnegative linearization of products.\\

It remains to show that $(\widetilde{P_n}(x))_{n\in\mathbb{N}_0}$ satisfies nonnegative linearization of products, too. We proceed similarly as above but give details because there are some relevant differences and we want to highlight the role of the various inequalities \eqref{eq:suf1}, \eqref{eq:suf2} and \eqref{eq:sufadd1}, \eqref{eq:sufadd2}. Note that the proof has not made use of \eqref{eq:sufadd2} so far. Let $N\in\mathbb{N}$ be arbitrary but fixed again.\\

\textit{Step 1:} we now show the existence of $t_2,t_4,\ldots,t_{2N}\in(0,\widetilde{\alpha_{2N}})$ with
\begin{equation*}
t_2=\widetilde{\alpha_{2N}}-\frac{\widetilde{\alpha_1}^2}{\widetilde{\alpha_{2N}}}
\end{equation*}
and, in the case of $N\geq2$,
\begin{equation}\label{eq:sufproof4Q}
t_{2n+2}=\widetilde{\alpha_{2N}}-\frac{\widetilde{\alpha_{2n+1}}^2t_{2n}}{\widetilde{\alpha_{2N}}t_{2n}-\widetilde{\alpha_{2n}}^2}
\end{equation}
and
\begin{equation}\label{eq:sufproof5Q}
\widetilde{\alpha_{2N}}^3c_{2n+1}<t_{2n+2}
\end{equation}
for all $n\in\{1,\ldots,N-1\}$. Set $t_2:=\widetilde{\alpha_{2N}}-\widetilde{\alpha_1}^2/\widetilde{\alpha_{2N}}$ and $t_4:=\widetilde{\alpha_{2N}}-\widetilde{\alpha_3}^2t_2/(\widetilde{\alpha_{2N}}t_2-\widetilde{\alpha_2}^2)$. By \eqref{eq:suf1} and Lemma~\ref{lma:sufficientcriterion2}, we have $t_2\in(0,\widetilde{\alpha_{2N}})$. The well-definedness of $t_4$ can be seen as follows: due to \eqref{eq:sufadd2}, we can estimate
\begin{equation*}
(1-s_4)(1-s_3)>s_1s_2+\frac{1-s_2+s_2s_3}{1-s_3}>s_1s_2+1-s_2.
\end{equation*}
The latter implies that
\begin{align*}
\widetilde{\alpha_4}^2-\widetilde{\alpha_1}^2-\widetilde{\alpha_2}^2&=a_4c_3-a_1-a_2c_1\\
&=(1-s_4)(1-s_3)-s_1-(1-s_2)(1-s_1)\\
&=(1-s_4)(1-s_3)-s_1s_2-1+s_2\\
&>0.
\end{align*}
Hence, via Lemma~\ref{lma:sufficientcriterion2} we obtain that
\begin{equation}\label{eq:sufproof6Q}
\widetilde{\alpha_{2N}}t_2-\widetilde{\alpha_2}^2=\widetilde{\alpha_{2N}}^2-\widetilde{\alpha_1}^2-\widetilde{\alpha_2}^2>0,
\end{equation}
and thus $t_4$ is well-defined. By construction, \eqref{eq:sufproof4Q} is true for $n=1$. We now show that also \eqref{eq:sufproof5Q} is true for $n=1$: in fact, we have
\begin{equation*}
t_4-\widetilde{\alpha_{2N}}^3c_3=\frac{(1-\widetilde{\alpha_{2N}}^2)(\widetilde{\alpha_{2N}}^2(c_3\widetilde{\alpha_{2N}}^2+a_1a_3c_2-a_1-a_2c_1-a_3c_2)+a_1a_3c_2)}{\widetilde{\alpha_{2N}}(\widetilde{\alpha_{2N}}^2-\widetilde{\alpha_1}^2-\widetilde{\alpha_2}^2)}
\end{equation*}
and
\begin{align*}
&c_3\widetilde{\alpha_{2N}}^2+a_1a_3c_2-a_1-a_2c_1-a_3c_2\\
&\geq c_3\widetilde{\alpha_4}^2+a_1a_3c_2-a_1-a_2c_1-a_3c_2\\
&=a_4c_3^2+a_1a_3c_2-a_1-a_2c_1-a_3c_2\\
&=(1-s_4)(1-s_3)^2+s_1s_2s_3-s_1-(1-s_2)(1-s_1)-s_2s_3\\
&>0
\end{align*}
due to \eqref{eq:sufadd2} and Lemma~\ref{lma:sufficientcriterion2}. Combined with \eqref{eq:sufproof6Q}, the latter estimation yields the assertion. In particular, $t_4$ is an element of $(0,\widetilde{\alpha_{2N}})$. The remaining part of this step works similarly as Step 1 above: let $n\in\{1,\ldots,N-2\}$ be arbitrary but fixed, and let $t_{2n}$ and $t_{2n+2}$ be such that \eqref{eq:sufproof4Q} and \eqref{eq:sufproof5Q} are satisfied for $n$. We define $t_{2n+4}:=\widetilde{\alpha_{2N}}-\widetilde{\alpha_{2n+3}}^2t_{2n+2}/(\widetilde{\alpha_{2N}}t_{2n+2}-\widetilde{\alpha_{2n+2}}^2)$ and can conclude from the induction hypothesis and \eqref{eq:sufproof1Q}, \eqref{eq:sufproof2Q} that $t_{2n+4}$ is a well-defined element of $(0,\widetilde{\alpha_{2N}})$ and that \eqref{eq:sufproof4Q} and \eqref{eq:sufproof5Q} are fulfilled for $n+1$.\\

\textit{Step 2:} since $t_2,t_4,\ldots,t_{2N}\in(0,\widetilde{\alpha_{2N}})$, we get
\begin{equation*}
t_{2n-1}:=\frac{\widetilde{\alpha_{2n-1}}^2}{\widetilde{\alpha_{2N}}-t_{2n}}>0
\end{equation*}
for all $n\in\{1,\ldots,N\}$, and the construction yields
\begin{equation*}
t_{n+1}=\widetilde{\alpha_{2N}}-\frac{\widetilde{\alpha_n}^2}{t_n}
\end{equation*}
for all $n\in\{1,\ldots,2N-1\}$. Furthermore, we have
\begin{equation*}
t_1=\frac{\widetilde{\alpha_1}^2}{\widetilde{\alpha_{2N}}-t_2}=\frac{\widetilde{\alpha_1}^2}{\widetilde{\alpha_{2N}}-\left(\widetilde{\alpha_{2N}}-\frac{\widetilde{\alpha_1}^2}{\widetilde{\alpha_{2N}}}\right)}=\widetilde{\alpha_{2N}}.
\end{equation*}
Therefore, we can apply Theorem~\ref{thm:ms01} (i) and Theorem~\ref{thm:el03} to conclude that $(\widetilde{P_n}(x))_{n\in\mathbb{N}_0}$ satisfies nonnegative linearization of products, too.
\end{proof}

\begin{remark}\label{rem:sufficientcriterion}
We want to make a comment concerning our constructions in the proof above.\footnote{In the following, we only refer to the construction concerning $(P_n(x))_{n\in\mathbb{N}_0}$, but the comment similarly applies to the construction concerning $(\widetilde{P_n}(x))_{n\in\mathbb{N}_0}$.} The reader might have wondered why we established the lower bound $t_{2n+1}>\alpha_{2N+1}^3a_{2n}$ provided by \eqref{eq:sufproof5P}, whereas $t_{2n+1}>0$ eventually would have been enough. The reason is that the induction hypothesis in Step 1 essentially benefited from the stronger estimation, or in other words: it was much more convenient to show an implication of the form $t_{2n+1}>\alpha_{2N+1}^3a_{2n}\Rightarrow t_{2n+3}>\alpha_{2N+1}^3a_{2n+2}$ instead of dealing with an implication of the form $t_{2n+1}>0\Rightarrow t_{2n+3}>0$. We also emphasize that the remark concerning the usage of computer algebra systems (Maple) at the end of Section~\ref{sec:intro} particularly refers to finding estimations like \eqref{eq:sufproof5P} and its counterpart \eqref{eq:sufproof5Q} via numerical experiments.
\end{remark}

\begin{proof}[Proof of Theorem~\ref{thm:actualresult}]
\begin{enumerate}[(i)]
\item A possible construction via Theorem~\ref{thm:sufficientcriterion} works as follows: consider
\begin{equation*}
s_n:=\frac{C^{\prime}}{K^{n-1}}
\end{equation*}
with $C^{\prime}:=\min\{C,1-C\}$. Then
\begin{align*}
&1-\frac{s_1}{1-s_1}-s_2=\frac{(1-2C^{\prime})K+(C^{\prime})^2-C^{\prime}}{(1-C^{\prime})K},\\
&\frac{s_{n-2}}{2}-2s_{n-1}-s_n=\frac{C^{\prime}(K^2-4K-2)}{2K^{n-1}}\;(n\geq3),
\end{align*}
and $K\in\mathbb{N}$ can be chosen such that $(s_n)_{n\in\mathbb{N}}\subseteq(0,1)$ and the right hand sides become positive. Now define $(c_n)_{n\in\mathbb{N}}$ by \eqref{eq:cnfirst} if $C>1/2$ and \eqref{eq:cnsecond} if $C<1/2$.
\item By \cite[Proposition 3.1]{KS25}, $0\in\widehat{\mathbb{N}_0}$ yields $c_1\leq1/2$. In the same way, $0\in\widetilde{\widehat{\mathbb{N}_0}}$ yields $a_1\leq1/2$ and therefore $c_1\geq1/2$. In conclusion, $0\in\widehat{\mathbb{N}_0}\cap\widetilde{\widehat{\mathbb{N}_0}}$ yields $c_1=1/2$, so the assertion follows from Theorem~\ref{thm:lasserobermaier}.
\item This is also a consequence of Theorem~\ref{thm:lasserobermaier}: let $(P_n(x))_{n\in\mathbb{N}_0}$ be an RWPS such that both $(P_n(x))_{n\in\mathbb{N}_0}$ and $(\widetilde{P_n}(x))_{n\in\mathbb{N}_0}$ satisfy nonnegative linearization of products. If $(P_n(x))_{n\in\mathbb{N}_0}=(T_n(x))_{n\in\mathbb{N}}$ (and hence $(\widetilde{P_n}(x))_{n\in\mathbb{N}_0}=(T_n(x))_{n\in\mathbb{N}}$), then $\widehat{\mathbb{N}_0}=\widetilde{\widehat{\mathbb{N}_0}}=[-1,1]$. Moreover, it is clear from \eqref{eq:threetermrec} and \eqref{eq:threetermrecswitched} that $|P_{2n}(0)|=\prod_{k=1}^n c_{2k-1}/a_{2k-1}$ and $P_{2n+1}(0)=0$ for all $n\in\mathbb{N}_0$, whereas $|\widetilde{P_{2n}}(0)|=\prod_{k=1}^n a_{2k-1}/c_{2k-1}$ and $\widetilde{P_{2n+1}}(0)=0$ for all $n\in\mathbb{N}_0$. Hence, if $c_{2n+1}>c_1>1/2$ for all $n\in\mathbb{N}$, then $0\in\widetilde{\widehat{\mathbb{N}_0}}$, whereas if $c_{2n+1}<c_1<1/2$ for all $n\in\mathbb{N}$, then $0\in\widehat{\mathbb{N}_0}$.
\end{enumerate}
\end{proof}

\begin{proof}[Proof of Theorem~\ref{thm:haarnotnondecreasing}]
\begin{enumerate}[(i)]
\item Assume that both $(P_n(x))_{n\in\mathbb{N}_0}$ and $(\widetilde{P_n}(x))_{n\in\mathbb{N}_0}$ satisfy nonnegative linearization of products and that both $h$ and $\widetilde{h}$ are increasing. By \eqref{eq:hdef}, this implies that $1-c_n=a_n\geq c_{n+1}$ and $c_n\geq a_{n+1}=1-c_{n+1}$ for all $n\in\mathbb{N}$, so $c_n+c_{n+1}\equiv1$. The latter, however, implies hat $(c_n)_{n\in\mathbb{N}}$ oscillates between two values. By Theorem~\ref{thm:lasserobermaier}, this yields $c_n\equiv1/2$ and consequently $(P_n(x))_{n\in\mathbb{N}_0}=(T_n(x))_{n\in\mathbb{N}_0}$. The converse is clear.
\item Let $(P_n(x))_{n\in\mathbb{N}_0}$ be as in Theorem~\ref{thm:sufficientcriterion} (explicit examples are given above). Then, by \eqref{eq:hdef} and Lemma~\ref{lma:sufficientcriterion},
\begin{equation*}
\frac{h(2n)}{h(2n-1)}=\frac{a_{2n-1}}{c_{2n}}=\frac{1-s_{2n-1}}{1-s_{2n}}<1
\end{equation*}
for all $n\in\mathbb{N}$ if $(c_n)_{n\in\mathbb{N}}$ is as in \eqref{eq:cnsecond}, whereas
\begin{equation*}
\frac{h(2n+1)}{h(2n)}=\frac{a_{2n}}{c_{2n+1}}=\frac{1-s_{2n}}{1-s_{2n+1}}<1
\end{equation*}
for all $n\in\mathbb{N}$ if $(c_n)_{n\in\mathbb{N}}$ is as in \eqref{eq:cnfirst}.
\end{enumerate}
\end{proof}

\begin{proof}[Proof of Theorem~\ref{thm:haar2smaller2}]
A possible construction via Theorem~\ref{thm:sufficientcriterion} works as follows: let $\epsilon\in(0,1)$, and let
\begin{equation*}
s_n:=\begin{cases} \frac{K^2}{(2+\epsilon)\cdot K^2-(1+\epsilon)}, & n=1, \\ \frac{1}{K^n}, & n\geq2. \end{cases}
\end{equation*}
Then
\begin{align*}
&1-\frac{s_1}{1-s_1}-s_2=\frac{\epsilon K^4-(1+\epsilon)(2K^2-1)}{(1+\epsilon)K^2(K^2-1)},\\
&\frac{s_1}{2}-2s_2-s_3=\frac{K^5+\mathcal{O}(K^3)}{2K^3((2+\epsilon)K^2-(1+\epsilon))},\\
&\frac{s_{n-2}}{2}-2s_{n-1}-s_n=\frac{K^2-4K-2}{2K^n}\;(n\geq4),
\end{align*}
and $K\in\mathbb{N}$ can be chosen in such a way that $(s_n)_{n\in\mathbb{N}}\subseteq(0,1)$ and the right hand sides become positive. Now if $(c_n)_{n\in\mathbb{N}}$ is given by \eqref{eq:cnsecond}, we have
\begin{equation*}
h(2)=\frac{1-s_1}{s_1(1-s_2)}=1+\epsilon.
\end{equation*}
\end{proof}

\begin{remark}
With regard to the comments at the end of Section~\ref{sec:main}, we observe that if $\epsilon\in(0,1)$ and $K$, $(s_n)_{n\in\mathbb{N}}$ and $(c_n)_{n\in\mathbb{N}}$ are as in the proof above, then
\begin{equation*}
h(1)=\frac{1}{s_1}>2
\end{equation*}
(Lemma~\ref{lma:sufficientcriterion}). Furthermore, one has $0\in\widehat{\mathbb{N}_0}$, cf. Theorem~\ref{thm:actualresult}.
\end{remark}

\section{Further properties of the new polynomials}\label{sec:further}

We consider the RWPS which fit in our sufficient criterion Theorem~\ref{thm:sufficientcriterion} in some more detail and particularly deal with the orthogonalization measures. We show the following: if $(P_n(x))_{n\in\mathbb{N}_0}$ and $(\widetilde{P_n}(x))_{n\in\mathbb{N}_0}$ are as in Theorem~\ref{thm:sufficientcriterion}, then there exist strictly increasing sequences $(x_n)_{n\in\mathbb{N}},(\widetilde{x_n})_{n\in\mathbb{N}}\subseteq[0,1)$ with $\lim_{n\to\infty}x_n=\lim_{n\to\infty}\widetilde{x_n}=1$ and
\begin{align*}
\mathrm{supp}\;\mu&=\{\pm1\}\cup\{\pm x_n:n\in\mathbb{N}\},\\
\mathrm{supp}\;\widetilde{\mu}&=\{\pm1\}\cup\{\pm\widetilde{x_n}:n\in\mathbb{N}\}.
\end{align*}
It clearly suffices to show the assertion for $(P_n(x))_{n\in\mathbb{N}_0}$. We proceed similarly (but with slight differences) to \cite[Remark 2 p. 427]{MS01} and consider the operator $T_1^2-\id$ on $\ell^2(h)$, which is given by
\begin{equation*}
(T_1^2-\id)f(m)=\begin{cases} a_1(f(2)-f(0)), & m=0, \\ a_2a_1(f(3)-f(1)), & m=1, \\ a_{m+1}a_m(f(m+2)-f(m))-c_m c_{m-1}(f(m)-f(m-2)), & m\geq2. \end{cases}
\end{equation*}
Since $\lim_{m\to\infty}a_{m+1}a_m=\lim_{m\to\infty}c_m c_{m-1}=0$ (recall from Lemma~\ref{lma:sufficientcriterion} that the sequence $(s_n)_{n\in\mathbb{N}}$ is a null sequence), $T_1^2-\id$ is compact. Hence, the (infinite) set
\begin{equation*}
\{x^2-1:x\in\mathrm{supp}\;\mu\}=\sigma(T_1^2-\id)
\end{equation*}
(which is a subset of $[-1,0]$) is of the form $\{0\}\cup\{y_n:n\in\mathbb{N}\}$ for some sequence $(y_n)_{n\in\mathbb{N}}\in[-1,0)$ with $\lim_{n\to\infty}y_n=0$. This yields the assertion.\\

Therefore, concerning the orthogonalization measures of RWPS such that both $(P_n(x))_{n\in\mathbb{N}_0}$ and $(\widetilde{P_n}(x))_{n\in\mathbb{N}_0}$ satisfy nonnegative linearization of products the examples provided by Theorem~\ref{thm:sufficientcriterion} are very different from the trivial example $(T_n(x))_{n\in\mathbb{N}_0}$ because for the latter one has an absolutely continuous measure with $\mathrm{supp}\;\mu=[-1,1]$. Nevertheless, the RWPS which fit in Theorem~\ref{thm:sufficientcriterion} share an important property with $(T_n(x))_{n\in\mathbb{N}_0}$, namely that the three dual objects coincide. More precisely, we have
\begin{equation*}
\mathrm{supp}\;\mu=\widehat{\mathbb{N}_0}=\mathcal{X}^b(\mathbb{N}_0)
\end{equation*}
and
\begin{equation*}
\mathrm{supp}\;\widetilde{\mu}=\widetilde{\widehat{\mathbb{N}_0}}=\widetilde{\mathcal{X}^b(\mathbb{N}_0)}.
\end{equation*}
This can be seen as follows: first, again it suffices to prove the assertion only with regard to $(P_n(x))_{n\in\mathbb{N}_0}$. Now we take an argument from the proof of \cite[Theorem 3.2 (iii)]{KS25} and \cite[Remark 3.3]{KS25} and proceed as follows (we restrict ourselves to a sketch and refer to \cite{KS25} for the details): motivated by \cite[Section 6]{MS01}, define $(R_n(x))_{n\in\mathbb{N}_0}\subseteq\mathbb{R}[x]$ via $R_n(x^2)=P_{2n}(x)$. $(R_n(x))_{n\in\mathbb{N}_0}$ fulfills $R_0(x)=1$, $R_1(x)=(x-c_1)/a_1$, $R_1(x)R_n(x)=a_n^R R_{n+1}(x)+b_n^R R_n(x)+c_n^R R_{n-1}(x)\;(n\in\mathbb{N})$ with
\begin{equation*}
a_n^R:=\frac{a_{2n}a_{2n+1}}{a_1},b_n^R:=\frac{a_{2n}c_{2n+1}+c_{2n}a_{2n-1}-c_1}{a_1},c_n^R:=\frac{c_{2n}c_{2n-1}}{a_1},
\end{equation*}
and as $(P_n(x))_{n\in\mathbb{N}_0}$ satisfies nonnegative linearization of products, $(R_n(x))_{n\in\mathbb{N}_0}$ satisfies nonnegative linearization of products, too. Since $\lim_{n\to\infty}a_n^R=\lim_{n\to\infty}c_n^R=0$ and $\lim_{n\to\infty}b_n^R=1$ (Lemma~\ref{lma:sufficientcriterion}), the resulting polynomial hypergroup is of `strong compact type' \cite[Proposition 4]{FLS05}. If $\mu_R$ denotes the corresponding orthogonalization measure, then
\begin{equation*}
\{1\}\cup\{x_n^2:n\in\mathbb{N}\}=\mathrm{supp}\;\mu_R=\left\{z\in\mathbb{C}:\max_{n\in\mathbb{N}_0}|R_n(z)|=1\right\},
\end{equation*}
where the first equality is clear from construction and the second equality follows from \cite[Theorem 2]{FLS05}. Hence, we get
\begin{equation*}
\mathcal{X}^b(\mathbb{N}_0)\subseteq\left\{z\in\mathbb{C}:\max_{n\in\mathbb{N}_0}|\underbrace{P_{2n}(z)}_{=R_n(z^2)}|=1\right\}=\{\pm1\}\cup\{\pm x_n:n\in\mathbb{N}\}=\mathrm{supp}\;\mu\subseteq\widehat{\mathbb{N}_0}\subseteq\mathcal{X}^b(\mathbb{N}_0).
\end{equation*}

\bibliography{jointworkjosefstefan}
\bibliographystyle{amsplain}

\end{document}